\documentclass[12pt,reqno]{amsart}

\newcommand\version{October 28, 2021}

\usepackage{amsmath,amsfonts,amsthm,amssymb,amsxtra}
\usepackage{bbm} 

\newtheorem{theorem}{Theorem}
\newtheorem{proposition}[theorem]{Proposition}
\newtheorem{lemma}[theorem]{Lemma}

\theoremstyle{definition}

\theoremstyle{remark}

\renewcommand{\epsilon}{\varepsilon}

\newcommand{\N}{\mathbb{N}}

\renewcommand{\phi}{\varphi}
\newcommand{\R}{\mathbb{R}}

\newcommand{\Sph}{\mathbb{S}}

\newcommand{\Z}{\mathbb{Z}}

\begin{document}

\title[Reverse Sobolev inequalities --- \version]{Reverse conformally invariant Sobolev inequalities on the sphere}

\author{Rupert L. Frank}
\address[Rupert L. Frank]{Mathematisches Institut, Ludwig-Maximilians-Universit\"at M\"unchen, Theresienstr. 39, 80333 M\"unchen, Germany, and Mathematics 253-37, Caltech, Pasa\-de\-na, CA 91125, USA}
\email{r.frank@lmu.de, rlfrank@caltech.edu}

\author{Tobias K\"onig}
\address[Tobias K\"onig]{Institut de Mathématiques de Jussieu -- Paris Rive Gauche	Université de Paris - Campus des Grands Moulins, Bâtiment Sophie Germain, Boite Courrier 7012,	8 Place Aurélie Nemours, 75205 Paris Cedex 13, France}
\email{koenig@imj-prg.fr}

\author{Hanli Tang}
\address[Hanli Tang]{Laboratory of Mathematics and Complex Systems (Ministry of Education), School of Mathematical Sciences, Beijing Normal University, Beijing, 100875, China}
\email{hltang@bnu.edu.cn}

\thanks{\copyright\, 2021 by the authors. This paper may be reproduced, in its entirety, for non-commercial purposes.\\
Partial support through US National Science Foundation grants DMS-1363432 and DMS-1954995 and Deutsche Forschungsgemeinschaft grant EXC-2111-390814868 (R.L.F.), Studienstiftung des deutschen Volkes and ANR BLADE-JC ANR-18-CE40-002 (T.K.) and National Natural Science Foundation of China (Grant No.11701032) and National Key Research and Development Program of China (Grant No. 2020YFA0712900) (H.T.) is acknowledged.}

\begin{abstract}
	We consider the optimization problem corresponding to the sharp constant in a conformally invariant Sobolev inequality on the $n$-sphere involving an operator of order $2s> n$. In this case the Sobolev exponent is negative. Our results extend existing ones to noninteger values of $s$ and settle the question of validity of a corresponding inequality in all dimensions $n\geq 2$.
\end{abstract}

\maketitle

\section{Introduction and main results}

We are interested in sharp constants in conformally invariant Sobolev inequalities. The classical version of this inequality concerns powers $(-\Delta)^s$ of the Laplacian in $\R^n$ with a real parameter $0<s<\frac n2$ and it reads
\begin{equation}
	\label{eq:sob}
	\left\|(-\Delta)^{s/2} U \right\|_2^2 \geq \mathcal S_{s,n} \| U\|_{\frac{2n}{n-2s}}^2
	\qquad\text{for all}\ U\in \dot H^s(\R^n)
\end{equation}
with
\begin{equation}
	\label{eq:sobconst}
	\mathcal S_{s,n} = (4\pi)^s \ \frac{\Gamma(\frac{n+2s}{2})}{\Gamma(\frac{n-2s}{2})} \left( \frac{\Gamma(\frac n2)}{\Gamma(n)} \right)^{2s/n}
	= \frac{\Gamma(\frac{n+2s}{2})}{\Gamma(\frac{n-2s}{2})} \ |\Sph^n|^{2s/n} \,.
\end{equation}
This inequality was proved in an equivalent, dual form by Lieb in \cite{Li}, where also the cases of equality were characterized. Moreover, in that work a fundamental property of \eqref{eq:sob}, namely its conformal invariance, was discovered and exploited. This result extends the earlier result in the local case $s=1$ going back to \cite{Rod,Ro,Au,Ta}.

Since $\R^n$ (or rather $\R^n\cup\{\infty\}$) and $\Sph^n$ are conformally equivalent, there is an equivalent version of \eqref{eq:sob} on $\Sph^n$. This form was found explicitly by Beckner in \cite[Eq.~(19)]{Be1993}, namely,
\begin{eqnarray}
	\label{eq:sobsphere}
	\left\| A_{2s}^{1/2} u \right\|_2^2 \geq \mathcal S_{s,n} \|u\|_{\frac{2n}{n-2s}}^2
	\qquad\text{for all}\ u\in H^s(\Sph^n)
\end{eqnarray}
with the same constant $\mathcal S_{s,n}$ as in \eqref{eq:sobconst} and with
\begin{align}
	\label{eq:opas}
	A_{2s} = \frac{\Gamma(B+\tfrac12 + s)}{\Gamma(B+\tfrac12 - s)}
	\qquad\text{and}\qquad
	B = \sqrt{-\Delta_{\Sph^n} + \tfrac{(n-1)^2}{4}} \,.
\end{align}
Note that the operators $B$ and $A_{2s}$ act diagonally in any basis of spherical harmonics, and on spherical harmonics of degree $\ell\in\N_0=\{0,1,2,\ldots\}$, the operator $B$ acts by multiplication with $\ell+\frac{n-1}{2}$ and, consequently, $A_{2s}$ acts by multiplication with
\begin{equation}
	\label{eq:alpha}
	\alpha_{2s,n}(\ell) = \frac{\Gamma(\ell+\frac n2 + s)}{\Gamma(\ell+\frac n2 - s)} \,.
\end{equation}
The operators $A_{2s}$ can be thought of as $(-\Delta_{\Sph^n})^{s}$ perturbed by lower order terms. For integer $s$, they are related to the GJMS operators in conformal geometry \cite{GrJeMaSp}. 

Note that as $s\nearrow \frac n2$, the integrability exponent $\frac{2n}{n-2s}$ in \eqref{eq:sob} and \eqref{eq:sobsphere} tends to $+\infty$. In \cite{Be1993} Beckner derived a conformally invariant endpoint inequality for $s=\frac n2$, which extends \cite{LeMi,On,OsPhSa}; see also \cite{CaLo} for an equivalent, dual inequality. In passing, we mention that in \cite{Be1992} Beckner also proved a conformally invariant endpoint inequality for $s=0$.

Our goal in this paper is to investigate the range
$$
s>\frac n2 \,.
$$
Note that in this case the integrability exponent $\frac{2n}{n-2s}$ in \eqref{eq:sobsphere} is negative, and therefore we will restrict ourselves to functions that are positive almost everywhere. It is because of this sign change that we call the inequalities in this paper `reverse' Sobolev inequalities.

The operators $A_{2s}$ are well-defined in the whole range $s>0$, provided one sets $\alpha_{2s,n}(\ell)=0$ when the denominator in \eqref{eq:alpha} has a pole. Note, however, that the operators $A_{2s}$ are no longer positive definite and therefore we define
$$
a_{2s}[u] := \sum_{\ell\in\N_0} \alpha_{2s,n}(\ell) \| P_\ell u\|_2^2
\qquad\text{for all}\ u \in H^s(\Sph^n) \,,
$$
where $P_\ell$ is the projection onto spherical harmonics of degree $\ell$. Note that when $s\leq\frac n2$, then $a_{2s}[u] = \| A_{2s}^{1/2} u \|_2^2$ for all $u \in H^s(\Sph^n)$. 

In the following we will study inequalities of the type
\begin{equation}
	\label{eq:sobsphererev}
	a_{2s}[u] \geq \mathcal S_{s,n} \left( \int_{\Sph^n} u^{-\frac{2n}{2s-n}} d\omega \right)^{-\frac{2s-n}{n}}
	\qquad\text{for all}\ 0<u\in H^s(\Sph^n) \,.
\end{equation}
We are interested in whether such an inequality holds with some finite constant $\mathcal S_{s,n}$ (not necessarily positive) and, if so, what the optimal value of this constant is.

A first inequality of this type, corresponding to $s=1$ in $n=1$, is shown in \cite{ExHaLo} and reads
$$
\int_{-\pi}^\pi \left( (u')^2 - \frac14 u^2\right)d\theta \geq -\pi^2 \left( \int_{-\pi}^\pi u^{-2}\,d\theta \right)^{-1}
\qquad\text{for all}\ u\in H^1(\R/2\pi\Z) \,.
$$
An independent proof of this inequality and a characterization of the cases of equality appears in \cite{AiChWe}. The case $s=2$ in $n=3$ is analyzed in \cite{YaZh}; see also \cite{HaYa}. The paper \cite{Ha} by Hang treats all cases $s\in\N\cap (\frac n2,\infty)$ in general dimensions $n\geq 1$ (here, $\N=\{1,2,3,\ldots\}$); for the cases $s=1,2$ in $n=1$, see also \cite{NiZh}. All these cases treated so far correspond to integer $s$, when $A_{2s}$ is a differential operator. 

In the above mentioned works it was established that inequality \eqref{eq:sobsphererev} is valid, with the constant given by \eqref{eq:sobconst}, when restricted to positive functions, provided that $s=\frac{n+1}{2}$, $\frac{n+3}{2}$ when $n$ is odd and $s=\frac n2+\N_0$ when $n$ is even. For odd $n$ equality is achieved exactly for the constant function, modulo conformal transformations, and for even $n$ exactly for positive linear combinations of spherical harmonics of degree $\leq s-\frac n2$. Moreover, a rather surprising result in \cite{Ha} is that for odd $n$ and $s\in \frac{n+5}{2}+\N_0$, the infimum
\begin{equation}
	\label{eq:infimum}
	I_{2s,n} := \inf_{0< u \in H^s(\Sph^n)} \left( \int_{\Sph^n} u^{-\frac{2n}{2s-n}} d\omega \right)^{\frac{2s-n}{n}} a_{2s}[u]
\end{equation}
is \emph{not} achieved and, in fact, there is not even a local minimum. As far as we know, this is one of the very few instances of conformally invariant functional inequalities on $\Sph^n$ without minimizers.

While the fundamental works \cite{ExHaLo,AiChWe,YaZh,HaYa,Ha}  answer many questions concerning the family of inequalities \eqref{eq:sobsphererev}, two natural ones remain open. (a) Do these results extend to all real values of the parameter $s>\frac n2$ and, if so, where does the transition between existence and nonexistence of a minimizer occur? (b) If there is no minimizer for \eqref{eq:infimum}, what is the value of the infimum?

In this paper we completely answer question (a) and, in dimension $n\geq 2$, also question (b). 

The following two theorems are our main results.

\begin{theorem}\label{main1}
	Let $n\geq 1$ and $s\in (\frac n2,\frac{n+4}2)\cup (\frac n2+ \N)$. Then for all $0\leq u\in H^s(\Sph^n)$ with $u^{-\frac{2n}{2s-n}}\in L^1(\Sph^n)$,
	\begin{equation}
		\label{eq:mainineq}
		a_{2s}[u] \left( \int_{\Sph^n} u^{-\frac{2n}{2s-n}}\,d\omega \right)^{\frac{2s-n}{n}} \geq \frac{\Gamma(\frac n2 + s)}{\Gamma(\frac n2 - s)}\ |\Sph^n|^\frac{2s}n \,.
	\end{equation}
	For $s\in (\frac n2,\frac{n+4}2)\setminus \{\frac{n+2}2\}$, equality is attained if and only if
	$$
	u(\omega) = c \left(1-\zeta\cdot\omega \right)^\frac{2s-n}{2}
	$$
	for some $c>0$ and $\zeta\in\R^{n+1}$ with $|\zeta|<1$. For $s\in\frac n2+ \N$, equality is attained if and only if $u$ is in the linear span of spherical harmonics of degree $\leq s - \frac n2$.
\end{theorem}

Note that the constant on the right side of \eqref{eq:mainineq} coincides with $\mathcal S_{s,n}$ in \eqref{eq:sobconst}. It is negative for $s\in (\frac n2,\frac{n+2}2)$, positive for $s\in (\frac {n+2}2,\frac{n+4}2)$ and zero for $s\in\frac n2+ \N$.

\begin{theorem}\label{main2}
	Let $n\geq 1$ and $s\in (\frac{n+4}2,\infty)\setminus (\frac n2+ \N)$. Then the infimum \eqref{eq:infimum} is not attained. If, in addition, $n\geq 2$, then $I_{2s,n}=-\infty$.
\end{theorem}

Theorem \ref{main1} for $s\in\frac n2+ \N$ is almost immediate from the definition of $a_{2s}$. In order to prove the theorem for $s\in (\frac n2,\frac{n+4}2)\setminus \{\frac{n+2}2\}$ we follow closely the strategy of Hang \cite{Ha}. Namely, first we prove existence of a minimizer and then we apply a result of Li \cite{Li04} characterizing all solutions to the corresponding Euler--Lagrange equation. In the proof of existence of minimizers one has to deal with the noncompact symmetry group of conformal transformations. To rule out loss of compactness modulo symmetries, an important role is played by the fact that $a_{2s}[u]\geq 0$ if $u$ vanishes at a point (together with its gradient if $s\in(\frac{n+2}{2},\frac{n+4}2)$). Similar results already appeared in \cite{ExHaLo,YaZh,HaYa,Ha}, where the authors dealt with local operators and could exploit integration by parts. In Proposition \ref{positivity} we prove the corresponding fact for general $s\in(\frac{n}{2},\frac{n+4}{2})$. We also proceed by going to $\R^n$, but the proof for noninteger $s$ is quite a bit more involved.

The first part of Theorem \ref{main2} follows again closely the strategy of Hang \cite{Ha} and also uses the result of Li \cite{Li04}. The second part answers a question that was left open in \cite{Ha} even in the integer case. The idea is to find a function $u\in H^s(\Sph^n)$ such that $u^{-\frac{2n}{2s-n}}\not\in L^1(\Sph^n)$ and $a_{2s}[u]<0$. Then using $u+\epsilon$ as trial functions for the infimum \eqref{eq:infimum} yields the assertion in Theorem \ref{main2}. The function $u$ that we choose vanishes to sufficiently high order on the equator $\{\omega_{n+1} =1\}$. Showing that the quadratic form is negative on this function, requires some rather explicit analysis involving spherical harmonics. The seemingly simpler question of whether $I_{2s,1}$ is finite or not for $n=1$, remains open.

\medskip

\emph{Background and open problems.} We end this introduction by putting our results into perspective and by mentioning some open problems.

\medskip

The work of Dou and Zhu \cite{DoZh} spiked a lot of interest in reversed Hardy--Littlewood--Sobolev (HLS) inequalities. The conformally invariant case of these inequalities states that for $\mu>0$ and nonnegative functions $F,G$ on $\R^n$,
\begin{equation}
	\label{eq:reversehls}
	\iint_{\R^n\times\R^n} F(x) |x-y|^\mu G(y) \,dx\,dy \geq \mathcal H_{\mu,n} \|F\|_{\frac{2n}{2n+\mu}} \|G\|_{\frac{2n}{2n+\mu}} \,.
\end{equation}
The optimal constant $\mathcal H_{\mu,n}>0$ and all optimizing functions $F,G$ were obtained in \cite{DoZh}; see also \cite{NgNg,ChLiLuTa}. By conformal invariance, \eqref{eq:reversehls} has an equivalent version on $\Sph^n$, namely,
\begin{equation}
	\label{eq:reversehlssphere}
	\iint_{\Sph^n\times\Sph^n} f(\omega) |\omega-\omega'|^\mu g(\omega') \,d\omega\,d\omega' \geq \mathcal H_{\mu,n} \|f\|_{\frac{2n}{2n+\mu}} \|g\|_{\frac{2n}{2n+\mu}} \,.
\end{equation}
For open questions in a non-conformally invariant case motivated by aggregation-diffusion equations, see \cite{CaDeDoFrHo,CaDeFrLe}.

While the (usual) HLS inequality studied in \cite{Li} is equivalent to the Sobolev inequality \eqref{eq:sob}, there seems to be no relation between \eqref{eq:reversehlssphere} and the inequality \eqref{eq:sobsphererev}. This is despite the fact that the integral kernel $|\omega-\omega'|^\mu$ appearing in \eqref{eq:reversehlssphere} is a multiple of the Green's function of the operator $A_{2s}$ with $\mu=2s-n$; see the proof of Lemma~\ref{confinv}. The fundamental difference between the (usual) HLS inequality and the reverse one is that the kernel is positive definite in the former case, but not in the latter. For inequalities \eqref{eq:reversehls} and \eqref{eq:reversehlssphere} optimizers exist for \emph{all} $\mu>0$ and one does not have an analogue of the nonexistence phenomenon in our Theorem \ref{main2}.

As we mentioned before, in our proof of Theorem \ref{main1} for $s\in (\frac n2,\frac{n+4}2)\setminus \{\frac{n+2}2\}$ we apply a result of Li \cite{Li04} and to do so, we use a relation between the Euler--Lagrange equations corresponding to \eqref{eq:reversehlssphere} and \eqref{eq:sobsphererev}. Interestingly, the analogue of this relation on Euclidean space may fail in the excluded case $s=\frac{n+2}{2}$; for an example with $n=2$ and $s=2$, see \cite{Zh}.

\medskip

Besides finding optimal constants and characterizing optimizers, a natural problem is to characterize all positive solutions of the corresponding Euler--Lagrange equation. For the Sobolev inequality \eqref{eq:sob} or, equivalently, for the corresponding HLS inequality, this was accomplished in \cite{ChLiOu}; see also \cite{Li04}. The latter paper also contains a characterization of solutions to the Euler--Lagrange equation corresponding to \eqref{eq:reversehls} and, in fact, just as in \cite{Ha} this will be a major ingredient in our proof of Theorem~\ref{main1}. For related classification results, see \cite{ChXu,HyWe,DoGuZh,FrKoTa} and references therein. In connection with this we emphasize that our Theorem \ref{main2} does not exclude that for $s\in(\frac 52,\infty)\setminus (\frac12 +\N)$ in $n=1$, the infimum in \eqref{eq:infimum} is attained for $u\in H^s(\Sph)$ with $u^{-\frac{2}{2s-1}}\in L^1(\Sph)$ and $\min u = 0$. We find it unlikely that such $u$ exist, but we cannot exclude their existence via \cite{Li04} since the Euler--Lagrange equation then only holds in $\Sph\setminus\{u=0\}$.

\medskip

After the works of Brezis and Lieb \cite{BrLi} and Bianchi and Egnell \cite{BiEg} and, in particular, in the last decade there has been an immense body of work concerning the quantitative stability of Sobolev and isoperimetric inequalities; see, e.g., \cite{FuMaPr,ChFrWe,CaFrLi,Fu,FiZh,BoDoNaSi,Fr} and references therein. It is natural to ask whether there are such stable versions of Theorem \ref{main1}. The computations with the linearization in the proof of Theorem \ref{main2} suggest that the answer is affirmative for $s\in(\frac n2,\frac{n+4}2)\setminus \{\frac{n+2}2\}$, but the precise form of such a purported inequality is unclear since the form $a_{2s}[u]$ is not positive semidefinite.

\medskip

Finally, we would like to mention the relation between the problem studied in this paper and conformal geometry. The sharp constant in the Sobolev inequality \eqref{eq:sob} with $s=1$ appears as a compactness threshold in the Yamabe problem on general manifolds \cite{Au}. The latter concerns the scalar curvature. Similarly, the case $s=2$ is related to the $Q$-curvature \cite{HaYa2} and generalized $Q$ curvatures were introduced in \cite{GrZw} for $0<s<\frac n2$; see also \cite{ChGo,CaCh,ChYa}. While (generalized) $Q$-curvature problems were originally considered for $s\leq \frac n2$, they are also meaningful for $s>\frac n2$ and, in fact, this was the original motivation for \cite{YaZh,HaYa}. Our Theorem \ref{main1} says that for $s\in (\frac n2,\frac{n+4}2)\cup (\frac n2 + \N)$, within the conformal class of the standard metric $g_{\Sph^n}$ on $\Sph^n$ and under the volume constraint $\mathrm{vol}_g(\Sph^n) = |\Sph^n|$, the standard metric maximizes the total generalized $Q$-curvature, defined by 
$$
Q_{2s,g} = -\frac{2}{2s-n}\, u^{\frac{2s+n}{2s-n}} A_{2s} u
\qquad\text{if}\ g = u^{-\frac{4}{2s-n}} g_{\Sph^n} \,.
$$
Our Theorem \ref{main1} plays the same role for the fractional order problems in \cite{ChGo,CaCh,ChYa} as the results in \cite{YaZh,HaYa} do in the $Q$-curvature problem on three-dimensional manifolds.


\section{Preliminaries}

\subsection{Conformal invariance}

In this subsection $n\geq 1$ and $s>\frac n2$ are fixed. Let $\Phi$ be a conformal transformation of $\Sph^n$ and, for a function $u$ on $\Sph^n$, set
$$
u_\Phi(\omega) = J_{\Phi}(\omega)^{-\frac{2s-n}{2n}} u(\Phi(\omega)) \,.
$$
Clearly, if $u$ is nonnegative and measurable, then
$$
\int_{\Sph^n} u_\Phi^{-\frac{2n}{2s-n}}\,d\omega = \int_{\Sph^n} u^{-\frac{2n}{2s-n}}\,d\omega \,.
$$

\begin{lemma}\label{confinv}
	If $u\in H^s(\Sph^n)$, then $u_\Phi\in H^s(\Sph^n)$ and
	$$
	a_{2s}[u_\Phi] = a_{2s}[u] \,.
	$$
\end{lemma}

\begin{proof}
	We prove the lemma under the assumption $s\not\in \frac n2+\N$, which implies the general result by a limiting argument. This assumption implies that $\alpha_{2s,n}(\ell)\neq 0$ for all $\ell\in\N_0$. Moreover, by Stirling's formula, $\alpha_{2s,n}(\ell)$ grows like $\ell^{2s}$. Thus, $A_{2s} = \Gamma(B+\frac12+s)/\Gamma(B+\frac12 -s)$ is invertible as an operator from $H^{-s}(\Sph^n)$ to $H^s(\Sph^n)$. The Funk--Hecke formula implies that if $Y$ is a spherical harmonic of degree $\ell\in\N_0$, then
	$$
	\int_{\Sph^n} |\omega - \omega'|^{2s-n} Y(\omega')\,d\omega' = \frac{2^{2s} \pi^\frac n2\,\Gamma(s)}{\Gamma(\frac n2 - s)} \frac{1}{\alpha_{2s,n}(\ell)} Y(\omega) \,;
	$$
	see \cite[Eq.~17]{Be1993} and also \cite[Cor.~4.3]{FrLi}. Consequently, $A_{2s}^{-1}$ is an integral operator with integral kernel
	$$
	\frac{\Gamma(\frac n2 - s)}{2^{2s} \pi^\frac n2\,\Gamma(s)}\ |\omega - \omega'|^{2s-n} \,.
	$$
	Using this formula, together with the fact that
	$$
	J_\Phi(\omega)^\frac 1n |\omega - \omega'|^2 J_\Phi(\omega')^\frac 1n = |\Phi(\omega) - \Phi(\omega')|^2 \,,
	$$
	we easily see that for any $v\in H^{-s}(\Sph^n)$
	$$
	A_{2s}^{-1} v^\Phi = ( A_{2s}^{-1} v )_\Phi \,,
	$$
	where we set
	$$
	v^\Phi(\omega) := J_{\Phi}(\omega)^{\frac{2s+n}{2n}} v(\Phi(\omega)) \,.
	$$
	This is equivalent to
	$$
	A_{2s} u_\Phi = (A_{2s} u)^\Phi \,.
	$$
	Multiplying this formula by $u_\Phi$ and integrating we obtain the claim.
\end{proof}


\subsection{Stereographic projection}

In the previous subsection we considered the behavior of $A_{2s}$ under a conformal transformation of $\Sph^n$. In this subsection we consider its behavior under stereographic projection. Throughout this subsection we fix $n\geq 1$ and $s\in(0,\infty)\setminus(\frac n2 + \N_0)$.

We introduce the (inverse) stereographic projection $\mathcal S:\R^n\to\Sph^n$ by
$$
\mathcal S_j(x) = \frac{2x_j}{1+|x|^2} \,,\ j=1,\ldots, n\,,
\qquad
\mathcal S_{n+1}(x) = \frac{1-|x|^2}{1+|x|^2} \,.
$$
Given a function $u$ on $\Sph^n$, we define two functions $u_\mathcal S$ and $u^\mathcal S$ on $\R^n$ by
\begin{equation}
	\label{eq:stereofcn}
	u_\mathcal S(x) = \left( \frac{1+|x|^2}{2} \right)^\frac{2s-n}{2} u(\mathcal S(x)) \,,
	\qquad
	u^\mathcal S(x) = \left( \frac{1+|x|^2}{2} \right)^{-\frac{2s+n}{2}} u(\mathcal S(x)) \,.
\end{equation}
Note that, since $(2/(1+|x|^2))^n$ is the Jacobian of $\mathcal S$, these formulas are similar to those appearing in Lemma \ref{confinv} and its proof.

\begin{lemma}\label{stereoinv}
	Let $s=N+\sigma$ with $N\in\N_0$ and $\sigma\in[0,1)$. If $n\geq 2$, then
	\begin{equation}
		\label{eq:confinvstereo}
		(-\Delta)^{-\sigma} \left( A_{2s} u \right)^\mathcal S = (-\Delta)^N u_\mathcal S
		\qquad\text{for all}\ u\in C^\infty(\Sph^n) \,,
	\end{equation}
	The same identity holds if $n=1$ and $\sigma\in[0,\frac12)$. If $n=1$ and $\sigma\in(\frac12,1)$, then
	\begin{equation}
		\label{eq:confinvstereo1}
		(-\tfrac{d^2}{dx^2})^{-\sigma+\frac12} H \left( A_{2s} u \right)^\mathcal S = \tfrac{d^{2N+1}}{dx^{2N+1}} u_\mathcal S
		\qquad\text{for all}\ u\in C^\infty(\Sph) \,,
	\end{equation}
	where $H$ is multiplication in Fourier space by $i\xi/|\xi|$.
\end{lemma}

 The proof will show that both sides of \eqref{eq:confinvstereo} and \eqref{eq:confinvstereo1} are continuous functions and that the identities hold pointwise.

We note that \eqref{eq:confinvstereo} and \eqref{eq:confinvstereo1} are precise versions of the `heuristic formula'
\begin{equation}
	\label{eq:confinvstereoheuristic}
	\left( A_{2s} u \right)^\mathcal S = (-\Delta)^s u_\mathcal S \,,
\end{equation}
which is analogous to the formula in the proof of Lemma \ref{confinv}. When trying to directly prove \eqref{eq:confinvstereoheuristic}, we ran into technical problems concerning the convolution of two tempered distributions. This can be circumvented by proving the less elegant formulas \eqref{eq:confinvstereo} and \eqref{eq:confinvstereo1}, which are just as good for our purposes.

\begin{proof}
	\emph{Step 1.} As a preparation we prove the following assertion, still assuming $s\in (0,\infty)\setminus (\frac n2+\N_0)$. If $n\geq 2$, then, for any measurable $f$ on $\R^n$ such that $|f(x)|\lesssim \langle x \rangle^{-2s-n}$,
	$$
	(-\Delta)^N \int_{\R^n} |x-x'|^{2s-n} f(x')\,dx' = \frac{2^{2s} \pi^\frac n2 \Gamma(s)}{\Gamma(\frac n2-s)}\, ((-\Delta)^{-\sigma} f)(x) \,.
	$$
	For $n=1$ and $\sigma\in[0,\frac12)$, the same assertion is true, while for $\sigma\in(\frac12,1)$ one has
	$$
	\frac{d^{2N+1}}{dx^{2N+1}} \int_{\R} |x-x'|^{2s-1} f(x')\,dx' = \frac{2^{2s} \pi^\frac 12 \Gamma(s)}{\Gamma(\frac 12-s)}\, \left( (-\tfrac{d^2}{dx^2})^{-\sigma+\frac12} H f\right)(x) \,.
	$$
	
	We prove this by induction on $N$. For $N=0$ and, if $n=1$, $s<\frac12$, this is a standard result; see, e.g., \cite[Theorem 5.9 and Corollary 5.10]{LiLo}. For $n=1$ and $\frac12<s<1$, using dominated convergence one easily sees that $x\mapsto \int_{\R} |x-x'|^{2s-1} f(x')\,dx'$ is $C^1$ and
	$$
	\frac{d}{dx} \int_{\R} |x-x'|^{2s-1} f(x')\,dx' = (2s-1) \int_\R |x-x'|^{2s-3} (x-x') f(x')\,dx' \,.
	$$
	Note that the integral kernel on the right side is locally integrable. The claimed identity then follows from the identity
	$$
	(2s-1) |x|^{2s-3}x = \frac{2^{2s} \pi^\frac 12 \Gamma(s)}{\Gamma(\frac 12-s)} \int_\R |\xi|^{-2s} i\xi e^{i\xi x}\,d\xi \,,
	$$
	where the right side exists as an improper Riemann integral. This identity can either be proved directly by moving the integration contour to the positive imaginary axis and using identities for the Gamma function, or by analytic continuation from the identity implicit in the above proof for $0<s<\frac12$.
	
	Now let us assume $N\geq 1$. Using dominated convergence, one easily verifies that $x\mapsto \int_{\R^n} |x-x'|^{2s-n} f(x')\,dx'$ is $C^2$ and that
	$$
	\Delta \int_{\R^n} |x-x'|^{2s-n} f(x')\,dx' = 4(s-\tfrac n2)(s-1) \int_{\R^n} |x-x'|^{2s-n-2} f(x')\,dx' \,.
	$$
	By induction, one concludes that, if either $n\geq 2$ or if $n=1$ and $\sigma<\frac12$,
	\begin{align*}
		(-\Delta)^N \int_{\R^n} \! |x-x'|^{2s-n} f(x')\,dx' & = - 4(s-\tfrac n2)(s-1) (-\Delta)^{N-1}\! \int_{\R^n} \! |x-x'|^{2s-n-2} f(x')\,dx' \\
		& = - 4(s-\tfrac n2)(s-1) \frac{2^{2(s-1)} \pi^\frac n2 \Gamma(s-1)}{\Gamma(\frac n2-s+1)}\, ((-\Delta)^{-\sigma}f)(x) \\
		& = \frac{2^{2s} \pi^\frac n2 \Gamma(s)}{\Gamma(\frac n2-s)}\, ((-\Delta)^{-\sigma}f)(x) \,.
	\end{align*}
	The proof for $n=1$ and $\frac12<\sigma<1$ is similar. This proves the claimed formula.
	
	\medskip
	
	\emph{Step 2.} It remains to prove \eqref{eq:confinvstereo} and \eqref{eq:confinvstereo1}. Let $u\in C^\infty(\Sph^n)$. Then $A_{2s} u\in C^\infty(\Sph^n)$ (indeed, for any $\sigma>0$, $u\in H^\sigma(\Sph^n)$, so $A_{2s} u \in H^{\sigma-2s}(\Sph^n)$) and, using the explicit integral kernel of $A_{2s}^{-1}$ from the proof of Lemma \ref{confinv},
	$$
	u(\omega) = \frac{\Gamma(\frac n2-s)}{2^{2s} \pi^\frac n2 \Gamma(s)} \int_{\Sph^n} |\omega - \omega'|^{2s-n} (A_{2s} u)(\omega')\,d\omega' \,.
	$$
	Thus, using
	$$
	\frac{1+|x|^2}2 \left|\mathcal S(x) - \mathcal S(x') \right|^2 \frac{1+|x'|^2}2 = \left| x- x'\right|^2
	$$
	and changing variables, we obtain
	$$
	u_\mathcal S(x) = \frac{\Gamma(\frac n2-s)}{2^{2s} \pi^\frac n2 \Gamma(s)} \int_{\R^n} |x-x'|^{2s-n} (A_{2s} u)^\mathcal S(x')\,dx' \,.
	$$
	Note that $|(A_{2s} u)^\mathcal S(x)| \leq \|A_{2s} u\|_\infty  (2/(1+|x|^2))^{(2s+n)/2}$, so, in particular, the integral on the right side converges absolutely. Now the result from Step 1 is applicable and we obtain \eqref{eq:confinvstereo} and \eqref{eq:confinvstereo1}. This concludes the proof of the proposition.
\end{proof}


\subsection{Positivity under a vanishing condition}

A crucial role in our proof of existence of a minimizer is played by the following 

\begin{proposition}\label{positivity}
	Let $n\geq 1$ and $s\in (\frac n2,\frac{n+4}2)$. For all $u\in H^s(\Sph^n)$ with $u(S)=0$ and, if $s>\frac{n+2}{2}$, $\nabla u(S)=0$, one has
	$$
	a_{2s}[u]\geq 0 \,.
	$$
	Moreover, if $s\in (\frac n2,\frac{n+2}2)$, then equality holds if and only if for some $c\in\R$,
	$$
	u(\omega) = c\, (1 + \omega_{n+1})^\frac{2s-n}{2}
	\qquad\text{for all}\ \omega\in\Sph^n \,.
	$$
	and if $s\in [\frac {n+2}2,\frac{n+4}2)$, then equality holds if and only if for some $c\in\R$, $b\in\R^n$,
	$$
	u(\omega) = c\, (1 + \omega_{n+1})^\frac{2s-n}{2} + (1 + \omega_{n+1})^\frac{2s-n-2}{2} b\cdot\omega'
	\qquad\text{for all}\ \omega=(\omega',\omega_{n+1})\in\Sph^n \,.
	$$
\end{proposition}

Here $S=(0,\ldots,0,-1)$ denotes the south pole.

\begin{proof}
	If $s=\frac{n+2}{2}$, then $\alpha_{2s,n}(\ell)\geq 0$ for all $\ell\in\N_0$ with equality if and only if $\ell\leq 1$. This proves immediately the claimed inequality as well as the characterization of the cases of equality. Thus, in the following we assume that $s\neq \frac{n+2}{2}$.
	
	First, let $u\in C^\infty_c(\Sph^n\setminus\{S\})$, so that $u_\mathcal S \in C^\infty_c(\R^n)$. If $n\geq 2$, or if $n=1$ and $s\in[1,\frac32)\cup[2,\frac52)$, we multiply \eqref{eq:confinvstereo} by $(-\Delta)^\sigma u_\mathcal S$ and integrate to obtain
	\begin{align}
		\label{eq:nonnegidentity}
		\int_{\R^n} |\xi|^{2s} |\widehat{u_\mathcal S}|^2\,d\xi & = \int_{\R^n} ((-\Delta)^\sigma u_\mathcal S)((-\Delta)^N u_\mathcal S)\,dx = \int_{\R^n} u_\mathcal S (A_{2s}u)^\mathcal S\,dx = \int_{\Sph^n} u A_{2s} u \,d\omega \notag \\
		& = a_{2s}[u] \,,
	\end{align}
	where the Fourier transform is defined by
	$$
	\widehat f(\xi) = (2\pi)^{-\frac n2} \int_{\R^n} e^{-i\xi\cdot x} f(x)\,dx \,.
	$$
	For $n=1$ and $s\in(\frac12,1)\cup(\frac32,2)$ we multiply \eqref{eq:confinvstereo1} by $(-\frac{d^2}{dx^2})^{\sigma-\frac12} \overline{H u_\mathcal S}$ and obtain the same identity \eqref{eq:nonnegidentity}. Since the left side in \eqref{eq:nonnegidentity} is nonnegative, we obtain the inequality in the first part of the proposition for $u\in C^\infty_c(\Sph^n\setminus\{S\})$. 
	
	We abbreviate 
	$$
	\mathcal Q:= 
	\begin{cases} 
	\{ u\in H^s(\Sph^n):\ u(S)=0\} & \text{if}\ s\in (\frac n2,\frac{n+2}2) \,,\\
	\{ u\in H^s(\Sph^n):\ u(S)=0,\, \nabla u(S) = 0\} & \text{if}\ s\in (\frac {n+2}2,\frac{n+4}2) \,.
	\end{cases}
	$$
	Our goal is to extend the identity \eqref{eq:nonnegidentity} to $u\in\mathcal Q$ and use this extension to characterize the cases of equality. It is well-known that the set $C_c^\infty(\Sph^n\setminus\{S\})$ is dense in $\mathcal Q$ with respect to the norm in $H^s(\Sph^n)$. Moreover, $a_{2s}$ is continuous with respect to the norm in $H^s(\Sph^n)$. This immediately implies that $a_{2s}[u]\geq 0$ for all $u\in\mathcal Q$.
	
	Let $u\in\mathcal Q$ and let $(u_j)\subset C^\infty_c(\Sph^n\setminus\{0\})$ be a sequence that converges to $u\in\mathcal Q$ in $H^s(\Sph^n)$. In particular, $(u_j)$ converges to $u$ in $L^2(\Sph^n)$ and, by a change of variables, $((u_j)_\mathcal S)$ converges to $u_\mathcal S$ in $L^2(\R^n,(2/(1+|x|^2))^{2s}\,dx)$. In particular, $((u_j)_\mathcal S)$ converges to $u_\mathcal S$ in the sense of tempered distributions, and therefore  $(\widehat{(u_j)_\mathcal S})$ converges to $\widehat{u_\mathcal S}$ in the sense of tempered distributions. On the other hand, the fact that $(u_j)$ is a Cauchy sequence in $H^s(\Sph^n)$, the identity \eqref{eq:nonnegidentity} and the $H^s$-continuity of $a_{2s}$ imply that $(\widehat{(u_j)_\mathcal S})$ is a Cauchy sequence in $L^2(\R^n,|\xi|^{2s}\,d\xi)$ and therefore convergent. A standard argument (namely, interlacing two Cauchy sequences) shows that the limit is independent of the approximating sequence. We deduce from this that the restriction of the distribution $\widehat{u_\mathcal S}$ to $\R^n\setminus\{0\}$ coincides with a function and that this function belongs to $L^2(\R^n,|\xi|^{2s}\,d\xi)$. Moreover, identity \eqref{eq:nonnegidentity} remains valid for $u\in\mathcal Q$, provided the integral on the left side is restricted to $\R^n\setminus\{0\}$ and $\widehat{u_\mathcal S}$ on the left side is interpreted as the restriction of the corresponding distribution to this set.
	
	In particular, if $a_{2s}[u]=0$ for some $u\in\mathcal Q$, then the distribution $\widehat{u_\mathcal S}$ vanishes on $\R^n\setminus\{0\}$ and therefore, by a well-known theorem about distributions, $\widehat{u_\mathcal S}$ coincides with a finite sum of derivatives of a Dirac delta distribution at the origin. Thus, $u_\mathcal S$ is a polynomial.
	
	By Morrey's inequality we have $u\in C^{s-\frac n2}(\Sph^n)$ if $s\in (\frac n2,\frac{n+2}2)$ and $u\in C^{1,s-\frac{n+2}2}(\Sph^n)$ if $s\in (\frac {n+2}2,\frac{n+4}2)$ and therefore in either case, the vanishing conditions imply that
	$$
	| u(\omega) | \lesssim |\omega - S|^{s-\frac n2}
	\qquad\text{for all}\ \omega\in\Sph^n \,,
	$$
	that is,
	\begin{equation}
		\label{eq:growthbound}
		|u_\mathcal S(x)| \lesssim \left( \frac{1+|x|^2}{2} \right)^{s-\frac n2} \left| \mathcal S(x)- S \right|^{s-\frac n2} = \left( 1+|x|^2 \right)^{\frac12(s-\frac n2)} 
		\qquad\text{for all}\ x\in\R^n \,.
	\end{equation}
	If $s<\frac{n+2}2$, then the right side grows sublinearly and therefore $u_\mathcal S$, being a polynomial, is equal to a constant $c$. Now $u_\mathcal S(x)=c$ is equivalent to $u(\omega) = c (1 + \omega_{n+1})^{(2s-n)/2}$, as claimed. If $s\in (\frac {n+2}2,\frac{n+4}2)$, then the right side in \eqref{eq:growthbound} grows subquadratically and therefore $u_\mathcal S$ is affine linear, that is, $u_\mathcal S(x) = c + b\cdot x$. This is equivalent to the form given in the proposition.
\end{proof}


\section{Proof of Theorem \ref{main1}}

In this section, we prove our first main theorem, whose nontrivial part says that for $s\in (\frac n2,\frac{n+4}2)\setminus \{\frac{n+2}2\}$ the infimum $I_{2s,n}$ in \eqref{eq:infimum} is achieved precisely by the constant function and its images under the group of conformal transformations. Similarly as in \cite{Ha} we proceed in two steps, namely first showing that the infimum is achieved and then characterizing the functions where the infimum is achieved.

\subsection{Existence of a minimizer}

\begin{proposition}\label{existence}
	Let $n\geq 1$ and $s\in (\frac n2,\frac{n+4}2)\setminus \{\frac{n+2}2\}$. Let $(u_j)\subset H^s(\Sph^n)$ be a sequence of nonnegative functions with $u_j^{-\frac{2n}{2s-n}} \in L^1(\Sph^n)$ and
	\begin{equation*}
		\lim_{j\to 0} a_{2s}[u_j] \left( \int_{\Sph^n} u_j^{-\frac{2n}{2s-n}}\,d\omega \right)^{\frac{2s-n}{n}}
		= I_{2s,n} \,.
	\end{equation*}
	Then there is a sequence $(\Phi_j)$ of conformal transformations of $\Sph^n$ and a sequence $(c_j)\subset\R_+$ such that, after passing to a subsequence, the functions $c_j (u_j)_{\Phi_j}$ converge in $H^s(\Sph^n)$ to an everywhere positive function that minimizes $I_{2s,n}$.
\end{proposition}

\begin{proof}
	\emph{Step 1.}
	After multiplying $u_j$ by a positive constant and after a rotation (which can be implemented as a conformal transformation), we may assume that for all $j$,
	$$
	\max_{\Sph^n} u_j =1
	\qquad\text{and}\qquad
	u_j(S)=\min_{\Sph^n} u_j \,.
	$$
	Here $S=(0,\ldots,0,-1)$ denotes the south pole and later $N=(0,\ldots,0,1)$ will denote the north pole

Let us show that $(u_j)$ is bounded in $H^s(\Sph^n)$. By the minimizing property, there is a $C\geq 0$ such that for all $j$,
\begin{equation}
	\label{eq:energybdd}
	a_{2s}[u_j] \left( \int_{\Sph^n} u_j^{-\frac{2n}{2s-n}}\,d\omega \right)^{\frac{2s-n}{n}} \leq C \,.
\end{equation}
Thus, by our normalization,
$$
a_{2s}[u_j] \leq C \left( \int_{\Sph^n} u_j^{-\frac{2n}{2s-n}}\,d\omega \right)^{-\frac{2s-n}{n}} \leq C |\Sph^n|^{-\frac{2s-n}{n}} \,.
$$
On the other hand, by Stirling's formula, $\alpha_{2s,n}(\ell)$ grows like $\ell^{2s}$. Since $\alpha_{2s,n}(\ell)\geq 0$ for all $\ell\geq 1$ if $s\in (\frac n2,\frac{n+2}2)$ and for all $\ell\geq 2$ if $s\in (\frac {n+2}2,\frac{n+4}2)$ and since the remaining finite rank terms are bounded in $L^2(\Sph^n)$, we see that for all $v\in H^s(\Sph^n)$,
\begin{equation}
	\label{eq:energyequivalence}
	a_{2s}[v] \geq c \| v \|_{H^s(\Sph^n)}^2 - C' \| v \|_{L^2(\Sph^n)}^2
\end{equation}
with $c>0$ and $C'<\infty$. Combining these inequalities we obtain
$$
c \| u_j \|_{H^s(\Sph^n)}^2 \leq C' \| u_j \|_{L^2(\Sph^n)}^2 + C |\Sph^n|^{-\frac{2s-n}{n}}
\leq C' |\Sph^n| + C |\Sph^n|^{-\frac{2s-n}{n}} \,,
$$
which proves the claimed boundedness.

Thus, after passing to a subsequence, we may assume that $(u_j)$ converges weakly in $H^s(\Sph^n)$ to some $u$. By Morrey's inequality and the Arzel\`a--Ascoli lemma, $(u_j)$ converges strongly to $u$ in $C(\Sph^n)$. As a consequence, $u\geq 0$ and
$$
\max_{\Sph^n} u = 1
\qquad\text{and}\qquad
u(S) = \min_{\Sph^n} u \,.
$$
We note that $a_{2s}$ is lower semicontinuous with respect to weak convergence in $H^s(\Sph^n)$. Indeed, this is clear for the positive part of the functional $a_{2s}$ and its negative part is finite rank and therefore continuous. As a consequence of lower semicontinuity,
$$
\liminf_{j\to\infty} a_{2s}[u_j] \geq a_{2s}[u] \,.
$$

\medskip

\emph{Step 2.}
If we have $u>0$ on $\Sph^n$, then $u_j^{-1}\to u^{-1}$ uniformly on $\Sph^n$ and consequently
$$
\int_{\Sph^n} u_j^{-\frac{2n}{2s-n}}\,d\omega \to \int_{\Sph^n} u^{-\frac{2n}{2s-n}}\,d\omega \,.
$$
This, together with the lower semicontinuity of $a_{2s}$ implies that $u$ is a minimizer. Moreover, setting $r_j:= u_j - u$ and using weak convergence in $H^s(\Sph^n)$, we find
$$
a_{2s}[u_j] = a_{2s}[u] + a_{2s}[r_j] + o(1)
$$
and therefore
\begin{align*}
	a_{2s}[u_j] \left( \int_{\Sph^n} u_j^{-\frac{2n}{2s-n}}d\omega \right)^{\frac{2s-n}{n}} & \!\!= a_{2s}[u] \left( \int_{\Sph^n} u^{-\frac{2n}{2s-n}}d\omega \right)^{\frac{2s-n}{n}} \!\!+ a_{2s}[r_j]\left( \int_{\Sph^n} u^{-\frac{2n}{2s-n}}d\omega \right)^{\frac{2s-n}{n}} \\
	& \quad + o(1) \,.
\end{align*}
Since the left side converges to $I_{2s,n}$ and the first term on the right side is equal to $I_{2s,n}$, we find that $a_{2s}[r_j]\to 0$. By \eqref{eq:energyequivalence} and the strong convergence of $r_j$ in $L^2(\Sph^n)$, we infer that $r_j\to 0$ strongly in $H^s(\Sph^n)$, that is, $u_j\to u$ in $H^s(\Sph^n)$, as claimed.

\medskip

\emph{Step 3.}
Thus, in what follows we assume that $\min u =0$. Our goal will be to show that after a conformal transformation and multiplication by a constant we can make the $u_j$ converge to a positive limit, which will be a minimizer.

We observe that
\begin{equation}
	\label{eq:negenergy}
	a_{2s}[u]\leq 0 \,.
\end{equation}
In fact, for $s\in(\frac n2,\frac{n+2}2)$ the infimum $I_{2s,n}$ is negative, as can be seen by taking a constant trial function, and therefore $a_{2s}[u_j]$ is negative for all sufficiently large $j$. Thus, \eqref{eq:negenergy} follows by lower semicontinuity. On the other hand, for $s\in(\frac {n+2}2,\frac{n+4}2)$ we have by Morrey's inequality $u\in C^{1,s-\frac{n+2}2}(\Sph^n)$. Since $u(S)=\min u = 0$ and $u\geq 0$ we have $\nabla u(S) =0$ and consequently,
$$
u(\omega) \leq C_u |\omega - S|^{s-\frac n2} 
\qquad\text{for all}\ \omega\in\Sph^n \,.
$$
Thus, by Fatou's lemma,
$$
\liminf_{j\to\infty} \int_{\Sph^n} u_j^{-\frac{2n}{2s-n}}\,d\omega \geq \int_{\Sph^n} u^{-\frac{2n}{2s-n}}\,d\omega  \geq C_u^{-\frac{2n}{2s-n}} \int_{\Sph^n} |\omega - S|^{- n}\,d\omega = \infty \,,
$$
that is, $\int_{\Sph^n} u_j^{-\frac{2n}{2s-n}}\,d\omega\to\infty$. Inserting this into \eqref{eq:energybdd} we obtain $\limsup_{j\to\infty} a_{2s}[u_j]\leq 0$ and then \eqref{eq:negenergy} follows again by lower semicontinuity.

On the other hand, by the first part of Proposition \ref{positivity}, the fact that $u(S)=0$ (and $\nabla u(S)=0$ if $s>\frac{n+2}{2}$) implies that $a_{2s}[u]\geq 0$ and therefore, in view of \eqref{eq:negenergy}, the second part of Proposition \ref{positivity} implies that
$$
u(\omega) = 2^{-\frac{2s-n}{2}} ( 1+\omega_{n+1})^\frac{2s-n}{2} 
\qquad\text{for all}\ \omega\in\Sph^n \,.
$$
(Here we used the normalization $\max u =1$ to determine the constant and, in case $s>\frac{n+2}{2}$ we use positivity of $u$ to deduce that $b=0$.)

With a sequence $(\lambda_j)\subset (0,\infty)$ to be determined later, we now consider the conformal transformation $\Phi_j$ of $\Sph^n$ given by $\Phi_j := \mathcal S \mathcal D_{\lambda_j} \mathcal S^{-1}$, where $\mathcal D_{\lambda_j}$ is dilation on $\R^n$ by $\lambda_j$, that is $(\mathcal D_{\lambda_j})(x) = \lambda_j x$, and $\mathcal S$ is the (inverse) stereographic projection. We set
$$
v_j(\omega) := J_{\Phi_j}(\omega)^{-\frac{2s-n}{2n}} u_j(\Phi_j(\omega))
$$
and
$$
\tilde u_j := \frac{v_j}{\max v_j} \,.
$$
By conformal invariance (Lemma \ref{confinv}) and homogeneity, $(\tilde u_j)$ is a minimizing sequence for $I_{2s,n}$ and it is normalized by $\max \tilde u_j = 1$. We argue as before and, after passing to a subsequence, we may assume that $(\tilde u_j)$ converges weakly in $H^s(\Sph^n)$ and strongly in $C(\Sph^n)$ to some $\tilde u$. 

Note that $\tilde u$ depends on the choice of the sequence $(\lambda_j)$. We claim that, for an appropriate choice of $(\lambda_j)$ (where for each $j$, $\lambda_j$ only depends on $u_j$), we have $\tilde u>0$. Once this is shown, we obtain in the same way as before that $\tilde u$ is a minimizer and that the convergence is strong in $H^s(\Sph^n)$, so we are done.

We argue by contradiction and assume that there is a $\xi\in\Sph^n$ such that $\tilde u(\xi)=0$. Then, still arguing as before, but with a rotated version of Proposition \ref{positivity},
\begin{equation}
	\label{eq:exproof2}
	\tilde u(\omega) = 2^{-\frac{2s-n}{2}} \left( 1 - \xi\cdot\omega \right)^\frac{2s-n}{2}
	\qquad\text{for all}\ \omega\in\Sph^n \,.
\end{equation}

We now compute explicitly
$$
J_{\Phi_j}(\omega) = \left( \frac{2\lambda_j}{1+\omega_{n+1} +\lambda_j^2 (1-\omega_{n+1})} \right)^n
$$
and, using $\Phi_j(N)=N$ and $\Phi_j(S)=S$, we obtain
$$
v_j(N)= \lambda_j^{-\frac{2s-n}{2}} u_j(N)
\qquad\text{and}\qquad
v_j(S) = \lambda_j^{\frac{2s-n}{2}} u_j(S) \,.
$$
Thus, if we choose
$$
\lambda_j := \left( \frac{u_j(N)}{u_j(S)} \right)^{\frac{1}{2s-n}} \,,
$$
then $\tilde u_j(S) = \tilde u_j(N)$ and, in the limit, $\tilde u(S)=\tilde u(N)$. By \eqref{eq:exproof2}, this implies that $\xi_{n+1}=0$ and $\tilde u(N)=2^{-\frac{2s-n}{2}}>0$.

Since $\min u_j = u_j(S)$, we have for any $\omega\in\Sph^n$,
\begin{align*}
	\tilde u_j(\omega) & \geq \frac{u_j(S)}{\max v_j} \left( \frac{1+\omega_{n+1} +\lambda_j^2 (1-\omega_{n+1})}{2\lambda_j} \right)^{\frac{2s-n}{2}} \\
	& = \tilde u_j(N) \left( \frac{\lambda_j^{-2} (1+\omega_{n+1}) + 1-\omega_{n+1}}{2} \right)^{\frac{2s-n}{2}}.
\end{align*}
Since $\tilde u_j(N)\to \tilde u(N) = 2^{-\frac{2s-n}{2}}$ and since $\lambda_j\to +\infty$ (because $u_j(S)\to u(S) = 0$ and $u_j(N)\to u(N)=1$), we infer that in the limit
$$
\tilde u(\omega) \geq  2^{2s-n} \left( 1- \omega_{n+1} \right)^{\frac{2s-n}{2}}.
$$
In particular, $\tilde u(\xi) \geq 2^{2s-n}>0$, a contradiction. This completes the proof.
\end{proof}

A variation of part of the above argument allows us to show the following.

\begin{lemma}\label{posmin}
	Let $n\geq 1$ and $s\in(\frac n2,\frac{n+4}2)\setminus\{\frac{n+2}{2}\}$. Assume that $u\in H^s(\Sph^n)$ is nonnegative with $u^{-\frac{2n}{2s-n}} \in L^1(\Sph^n)$ and
	\begin{equation*}
		a_{2s}[u] \left( \int_{\Sph^n} u^{-\frac{2n}{2s-n}}\,d\omega \right)^{\frac{2s-n}{n}}
		= I_{2s,n} \,.
	\end{equation*}
	Then $u$ is everywhere positive.
\end{lemma}

\begin{proof}
	We argue by contradiction and assume that $\min u=0$. Then, by a rotated version of Proposition \ref{positivity}, $a_{2s}[u]\geq 0$. Since $a_{2s}[1]<0$ for $s\in(\frac n2,\frac{n+2}2)$, we immediately obtain a contradiction in that case. On the other hand, if $s\in(\frac{n+2}2,\frac{n+4}2)$, then similarly as in the previous proof $u(\omega) \lesssim |\omega - \xi|^{s-\frac n2}$ for all $\omega\in\Sph^n$ and some $\xi\in\Sph^n$ and consequently $u^{-\frac{2n}{2s-n}}\not\in L^1(\Sph^n)$, which is again a contradiction.	
\end{proof}


\subsection{Proof of Theorem \ref{main1}}

First, let $s\in\frac{n}{2}+\N$. Then $\alpha_{2s,n}(\ell)\geq 0$ for all $\ell\in\N_0$ with equality if and only if $\ell\leq s-\frac n2$. This immediately proves the inequality and the characterization of cases of equality.

In the remainder of this proof, let $s\in(\frac n2,\frac{n+4}{2})\setminus\{\frac{n+2}{2}\}$. Then, according to Proposition \ref{existence}, there is a minimizer $u$ for the infimum $I_{2s,n}$. Conversely, assume that $0\leq u\in H^s(\Sph^n)$ with $u^{-\frac{2n}{2s-n}}\in L^1(\Sph^n)$ realized equality in \eqref{eq:mainineq}. Then by Lemma \ref{posmin} $u>0$ and using this, it is easy to derive the Euler--Lagrange equation
$$
A_{2s} u = \lambda\, u^{-\frac{2s+n}{2s-n}} \,,
\qquad
\lambda = \frac{a_{2s}[u]}{\int_{\Sph^n} u^{-\frac{2n}{2s-n}}\,d\omega} \,.
$$
(The form of the Euler--Lagrange multiplier follows by integrating the equation against $u$.) The equation holds a-priori in $H^{-s}(\Sph^n)$, but since the right side is square-integrable (in fact, H\"older continuous), a standard bootstrap argument yields that $u\in C^\infty(\Sph^n)$. Applying the inverse $A_{2s}^{-1}$ to both sides as in the proof of Lemma \ref{confinv}, we find
$$
u = \lambda\, A_{2s}^{-1} u^{-\frac{2s+n}{2s-n}} = \lambda\, \frac{\Gamma(\frac n2-s)}{2^{2s} \pi^\frac n2 \Gamma(s)} \int_{\Sph^n} |\omega - \omega'|^{2s-n} u(\omega')^{-\frac{2s+n}{2s-n}} \,d\omega' \,.
$$
Taking into account the sign of $\Gamma(\frac n2-s)$, we see that $\lambda<0$ if $s\in (\frac n2,\frac{n+2}{2})$ and $\lambda>0$ if $s\in(\frac {n+2}2,\frac{n+4}{2})$. Defining $u_\mathcal S$ and $u^\mathcal S$ as in \eqref{eq:stereofcn} and arguing as in Step 2 of the proof of Lemma \ref{stereoinv}, we obtain
\begin{align*}
	u_\mathcal S(x) & = \lambda\, \frac{\Gamma(\frac n2-s)}{2^{2s} \pi^\frac n2 \Gamma(s)} \int_{\R^n} |x - x'|^{2s-n} (u^{-\frac{2s+n}{2s-n}})^\mathcal S(x')\,dx' \\
	& = \lambda\, \frac{\Gamma(\frac n2-s)}{2^{2s} \pi^\frac n2 \Gamma(s)} \int_{\R^n} |x - x'|^{2s-n} (u_\mathcal S(x'))^{-\frac{2s+n}{2s-n}}\,dx' \,.
\end{align*}
Applying \cite[Theorem 1.5]{Li04} to a suitable multiple of $u_\mathcal S$ (at this point we use the sign of $\lambda$), we find that, for some $a\in\R^n$, $b,c>0$,
$$
u_\mathcal S(x) = c \left( \frac{b^2 + |x-a|^2}{2b} \right)^{\frac{2s-n}{2}}.
$$
This means that, with $\zeta := (2\eta - b^2(1 + \eta_{n+1})e_{n+1})/(2 + b^2(1 + \eta_{n+1}))$ and $\eta := S(a)$,
$$
u(\omega) = c \left( \frac{1 - \zeta\cdot\omega}{\sqrt{1-|\zeta|^2}} \right)^\frac{2s-n}{2},
$$
which is of the form claimed in the theorem. 

Conversely, if $u$ is of the form in the theorem, then $u=c J_\Phi^{-(2s-n)/(2n)} = c 1_ \Phi$ for some conformal transformation $\Phi$ of $\Sph^n$ and some $c>0$. (This follows, for instance, by reversing the above computation, namely by showing that $u_\mathcal S$ is a multiple of a translation and dilation of $((1+|x|^2)/2)^{(2s-n)/2}$.) Then $\int u^{-2n/(2s-n)}\,d\omega = c^{-2n/(2s-n)} \int 1^{-2n/(2s-n)}\,d\omega$ and, by Lemma \ref{confinv}, $a_{2s}[u] = c^2 a_{2s}[1]$. In particular, the value of the left side of \eqref{eq:mainineq} is independent of $\zeta$ and $c$ and since, according to what we showed before, the left side is minimal for some $\zeta$ and $c$, it is in fact minimal for every $\zeta$ and $c$. This concludes the proof of Theorem \ref{main1}.


\section{Proof of Theorem \ref{main2}}

In this section we prove our second main result, which says that for $s\in(\frac{n+4}{2},+\infty)\setminus (\frac n2 + \N)$ there is no minimizer for the infimum $I_{2s,n}$ in \eqref{eq:infimum} and that, at least for $n\geq 2$, one has $I_{2s,n}=-\infty$. These two assertions are proved in the following two subsections.

\subsection{Local instability}

We prove more than what is stated in Theorem \ref{main2}, namely that the quantity in \eqref{eq:infimum} does not even have a local minimizer $0<u\in H^s(\Sph^n)$. Indeed, if such a local minimizer would exist, we could repeat the argument in the proof of Theorem \ref{main1} and would infer that this minimizer is necessarily of the form $c 1_\Phi$ for a conformal transformation $\Phi$ of $\Sph^n$ and a constant $c>0$. Since the minimization problem is homogeneous and conformally invariant (Lemma \ref{confinv}), it therefore suffices to show that the constant function $1$ is not a local minimizer. We do this by showing that the second variation is not positive semidefinite.

A simple computation shows that for every $\phi\in H^s(\Sph^n)$, as $t\to 0$,
$$
a[1+t\phi] \left( \int_{\Sph^n} (1+t\phi)^{-\frac{2n}{2s-n}}d\omega \right)^\frac{2s-n}{n}
= a[1] |\Sph^n|^{\frac{2s-n}n} + t^2 |\Sph^n|^\frac{2s-n}{n} H[\phi] + o(t^2)
$$
with
\begin{align*}
	H[\phi] & := a_{2s}[\phi]  + \frac{2s+n}{2s-n}\, \frac{a_{2s}[1]}{|\Sph^n|} \int_{\Sph^n} \phi^2\,d\omega  + \frac{4(s-n)}{2s-n}\, \frac{a_{2s}[1]}{|\Sph^n|^2} \left( \int_{\Sph^n} \phi\,d\omega \right)^2 \\
	& \quad - 4\, \frac{a_{2s}[1,\phi]}{|\Sph^n|} \int_{\Sph^n} \phi\,d\omega \,.
\end{align*}
Here $a_{2s}[\cdot,\cdot]$ is the natural bilinear form associated to $a_{2s}[\cdot]$. This can be rewritten as
\begin{align*}
	H[\phi] & = a_{2s}[\phi]  + \frac{2s+n}{2s-n}\, \alpha_{2s,n}(0) \int_{\Sph^n} \phi^2\,d\omega  - \frac{4s}{2s-n}\, \alpha_{2s,n}(0) \left( \int_{\Sph^n} \phi\,d\omega \right)^2.
\end{align*}

If $2k<s-\frac n2<2k+1$ for some $k\in\N$, we choose $\phi$ to be an $L^2$-normalized spherical harmonic of degree two and obtain
$$
H[\phi] = \alpha_{2s,n}(2) + \frac{2s+n}{2s-n} \alpha_{2s,n}(0) = \frac{\Gamma(2+\frac n2+s)}{\Gamma(2+\frac n2-s)} + \frac{2s+n}{2s-n} \, \frac{\Gamma(\frac n2+s)}{\Gamma(\frac n2-s)} = 2s \frac{\Gamma(1+\frac n2+s)}{\Gamma(2+\frac n2-s)}.
$$
Since $\Gamma(2+\frac n2-s)<0$, we have $H[\phi]<0$, which shows the instability of $1$.

If $2k+1<s-\frac n2<2(k+1)$ for some $k\in\N$, we choose $\phi$ to be an $L^2$-normalized spherical harmonic of degree three and obtain
\begin{align*}
	H[\phi] & = \alpha_{2s,n}(3) + \frac{2s+n}{2s-n} \alpha_{2s,n}(0) = \frac{\Gamma(3+\frac n2+s)}{\Gamma(3+\frac n2-s)} + \frac{2s+n}{2s-n} \, \frac{\Gamma(\frac n2+s)}{\Gamma(\frac n2-s)} \\
	& = 2s(n+3) \frac{\Gamma(1+\frac n2+s)}{\Gamma(3+\frac n2-s)}.
\end{align*}
Since $\Gamma(3+\frac n2-s)<0$, we have $H[\phi]<0$, which shows again the instability of $1$.

\subsection{Global instability}

We now complete the proof of Theorem \ref{main2} by showing that 
\begin{equation}
	\label{I = -infty}
	I_{2s,n} = -\infty \qquad \text{ if } n \geq 2 \quad \text{ and } \quad s\in (\tfrac{n+4}2,\infty)\setminus (\tfrac n2+ \N). 
\end{equation}
We will give a separate proof for the following two subcases
\begin{align}
	2K< s- \frac n2 <2K+1 & \quad \text{for some}\ K\in\N\,, \label{even} \\
	2K+1 < s - \frac n2 < 2K+2 & \quad \text{for some}\ K\in\N \,. \label{odd}
\end{align}
Note that in the first case, we have $\alpha_{2s,n}(2k)<0$ for all $k=0,\ldots,K$ and $\alpha_{2s,n}(\ell)>0$ for all other $\ell$. In the second case, we have $\alpha_{2s,n}(2k+1)<0$ for all $k=0,\ldots,K$ and $\alpha_{2s,n}(\ell)>0$ for all other $\ell$.

For both cases, we will give a test function $u \geq 0$ such that 
\begin{equation}
	\label{u test function}
	-\infty < a_{2s}[u]<0
	\qquad\text{and}\qquad
	\int_{\Sph^n} u^{-\frac{2n}{2s-n}}\,d\omega = + \infty \,.
\end{equation}
This essentially proves \eqref{I = -infty}, except that the function may not satisfy the strict inequality $u>0$. But we can simply take $u+\epsilon$ with a constant $\epsilon>0$ as a trial function for the infimum and let $\epsilon\to 0_+$.

\subsubsection*{The case \eqref{even}}

Let $K \in \N$ be as in \eqref{even}. We shall show that \eqref{u test function} holds for the function 
\[ u(\omega) = \omega_{n+1}^{2K}. \]

The integral condition is simple. Using spherical coordinates with $\omega_{n+1}=\cos\theta$, $\theta\in[0,\pi]$, and changing variables $t=\cos\theta\in[-1,1]$ we obtain
$$
\int_{\Sph^n} u^{-\frac{2n}{2s-n}}\,d\omega = |\Sph^{n-1}| \int_0^\pi |\cos\theta|^{-\frac{4nK}{2s-n}} \sin^{n-1}\theta\,d\theta = |\Sph^{n-1}| \int_{-1}^1 |t|^{-\frac{4nK}{2s-n}} (1-t^2)^\frac{n-2}{2}\,dt \,.
$$
This integral is divergent if and only if $\frac{4nK}{2s-n} \geq 1$. In view of \eqref{even}, this is the case if and only if $n \geq 2$. 

Next, we show that $a_{2s}[u]<0$. We claim that
\begin{equation}
	\label{u expansion}
	u(\omega) = \sum_{k=0}^{K} c_k C_{2k}^{(\frac{n-1}{2})}(\omega_{n+1}) \,,
\end{equation}
where $C_\ell^{(\alpha)}$ are the Gegenbauer (or ultraspherical) polynomials and where $c_k\in\R$. It is well known (see, for instance, \cite[Thm. IV.2.14]{StWe}) that the function $C_{\ell}^{(\frac{n-1}{2})}(\omega_{n+1})$ is a spherical harmonic of degree $\ell$, namely a so-called zonal harmonic. 

Note that we claim that in the spherical harmonic expansion \eqref{u expansion} of $u$ there are only terms of even degree at most $K$. As noted above, the condition \eqref{even} then guarantees that $\alpha_{2s,n}(2k) < 0$ for all $k = 0,...,K$ and thus
$$
a_{2s}[u] = \sum_{k=0}^K \alpha_{2s,n}(2k)\, c_k^2 \|Y_{2k}\|_{L^2(\mathbb S^n)}^2 < 0 \,,
$$
as desired.

We recall two standard facts about the Gegenbauer polynomials. First, $C_\ell^{(\alpha)}$ is a polynomial of exact degree $\ell$ and second, $C_\ell^{(\alpha)}$ has the same parity as $\ell$, see \cite[Eq. 18.5.10]{DLMF}. That is, for every $k$, 
\[ C_{2k}^{(\frac{n-1}{2})}(t) = a_{k,k} t^{2k} + a_{k,k-1} t^{2k-2} + \ldots + a_{k,0} \]
with $a_{k,k}\neq 0$. 
Thus, the desired expansion
\begin{equation}
	\label{t^2K expansion}
	t^{2K} = \sum_{k=0}^K c_k\, C_{2k}^{(\frac{n-1}{2})}(t)
	\qquad\text{for all}\ t\in[-1,1] \,.
\end{equation}
can be equivalently rewritten, with respect to the basis $\{ t^{2K}, t^{2K-2}, \ldots, t^2, 1 \}$ of even polynomials on $[-1,1]$ of order at most $2K$, as 
\begin{equation*}
	\left(
	\begin{array}{cccc}
		a_{ K,K}& 0 &...  & 0\\
		a_{K,K-1}& a_{K-1,K-1} &...  & 0\\
		&  &...& \\
		a_{K,0}& a_{K-1,0} &...  &a_{0,0}
	\end{array}
	\right )
	\left(
	\begin{array}{cccc}
		c_{K}\\
		c_{K-1}\\
		... \\
		c_{0}
	\end{array}
	\right )
	=
	\left(
	\begin{array}{cccc}
		1\\
		0\\
		... \\
		0
	\end{array}
	\right ).
\end{equation*}
Since the matrix on the left side is of lower-triangular form with non-zero diagonal entries, its determinant is non-zero. Hence there are (unique) numbers $c_0,...,c_K \in \R$ such that \eqref{t^2K expansion}, and hence \eqref{u expansion}, holds. This completes the proof in the case \eqref{even}. 


\subsubsection*{The case \eqref{odd}}

In this case, the above argument becomes a bit more involved because we need to work with the more complicated test function
\[ u(\omega) := \omega_{n+1}^{2K} - \omega_{n+1}^{2K+1}, \]
where $K$ is as in \eqref{odd}. Indeed, since $\alpha_{2s,n}(\ell) < 0$ if and only if $\ell = 1,3,...,2K+1$, the test function needs to contain an `odd' component like $\omega_{n+1}^{2K+1}$ to achieve $a_{2s}[u]<0$. On the other hand, the `even' term $\omega_{n+1}^{2K}$ is necessary to ensure $u \geq 0$. 

Let us verify divergence of the integral. With the same change of variables as before we obtain
\begin{align*}
	\int_{\Sph^n} u^{-\frac{2n}{2s-n}}\,d\omega & = |\Sph^{n-1}| \int_0^\pi |\cos\theta|^{-\frac{4nK}{2s-n}} ( 1- \cos\theta)^{-\frac{2n}{2s-n}} \sin^{n-1}\theta \,d\theta \\
	& = |\Sph^{n-1}| \int_{-1}^1 |t|^{-\frac{4nK}{2s-n}} ( 1- t)^{-\frac{2n}{2s-n}} (1-t^2)^{\frac{n-2}{2}} \,dt \,.
\end{align*}
The integral is divergent (at $t=0$) if and only if $\frac{4nK}{2s-n}\geq 1$. In view of \eqref{odd}, this is the case if and only if $n\geq 2$.

Next, we show that $a_{2s}[u]< 0$. By using the properties of Gegenbauer polynomials as in the case \eqref{even}, we find  
\begin{equation}
	\label{t^2K, t^2K+1 expansion}
	t^{2K} = \sum_{k= 0}^K c_k C_{2k}^{(\frac{n-1}{2})} (t) \qquad \text{ and } \qquad t^{2K+1} = \sum_{k= 0}^K d_k C_{2k+1}^{(\frac{n-1}{2})} (t) 
\end{equation} 
for suitable coefficients $c_k, d_k \in \R$. (In fact, in this case we shall need to find $c_k$ and $d_k$ explicitly, see \eqref{c_k d_k relation} and Lemma \ref{lemma gegenbauer} below.)  Therefore 
\[ u(\omega) = \sum_{k=0}^K \left( c_k C_{2k}^{(\frac{n-1}{2})} (\omega_{n+1}) - d_k C_{2k+1}^{(\frac{n-1}{2})} (\omega_{n+1}) \right). \]
Since $a_{2s}$ is diagonal with respect to spherical harmonics, we thus obtain
\begin{align*}
	a_{2s,n}[u] =  \! \sum_{k=0}^K \! \left( \alpha_{2s,n}(2k) c_k^2 \|C_{2k}^{(\frac{n-1}{2})} (\omega_{n+1})\|_{L^2(\mathbb S^n)}^2 \!+ \alpha_{2s,n}(2k+1) d_k^2 \|C_{2k+1}^{(\frac{n-1}{2})} (\omega_{n+1})\|_{L^2(\mathbb S^n)}^2 \right)\!.
\end{align*}
By \eqref{eq:alpha}, we have the relation $\alpha_{2s,n}(2k+1) = - \frac{2s + n + 4k}{2s-n - 4k} \alpha_{2s,n}(2k)$, where $\frac{2s + n + 4k}{2s-n - 4k} > 0$ for all $k = 0,...,K$, thanks to \eqref{odd}. Hence 
\begin{equation}
	\label{a(u) difference}
	a_{2s,n}[u] = \!\sum_{k=0}^K \! \alpha_{2s,n}(2k) \! \left( \! c_k^2 \|C_{2k}^{(\frac{n-1}{2})} (\omega_{n+1})\|_{L^2(\mathbb S^n)}^2 - \frac{2s + n + 4k}{2s-n - 4k} d_k^2 \|C_{2k+1}^{(\frac{n-1}{2})} (\omega_{n+1})\|_{L^2(\mathbb S^n)}^2 \! \right)\!. 
\end{equation} 
In view of $\alpha_{2s,n}(2k) > 0$ for all $k=0,...,K$ by \eqref{odd}, the desired inequality $a_{2s}[u] < 0$ follows if we can show that the difference in brackets is strictly negative for every $k = 0,...,K$. To do so, we claim that the coefficients $c_k$ and $d_k$ are related by
\begin{equation}
	\label{c_k d_k relation}
	d_k = c_k \frac{(2K + 1) (4k + n +1)}{(2K + 2k + n + 1)(4k + n -1)}
\end{equation}
and the $L^2$-norms of the spherical harmonics by
\begin{equation}
	\label{L^2 norms relation}
	\|C_{2k+1}^{(\frac{n-1}{2})} (\omega_{n+1})\|_{L^2(\mathbb S^n)}^2 = \|C_{2k}^{(\frac{n-1}{2})} (\omega_{n+1})\|_{L^2(\mathbb S^n)}^2 \frac{(2k + n -1)(4k + n - 1)}{(2k+1)(4k+n + 1)}. 
\end{equation}
We defer the details of these computations to Lemma \ref{lemma gegenbauer} below.

Inserting \eqref{c_k d_k relation} and \eqref{L^2 norms relation} into \eqref{a(u) difference}, the inequality we need to verify reduces to 
\begin{equation}
	\label{ineq}
	1 - \frac{2s + n + 4k}{2s-n - 4k} \cdot \frac{(2K + 1)^2 (2k + n -1) (4k + n +1)}{(2K + 2k + n + 1)^2(2k + 1) (4k + n -1)}  < 0. 
\end{equation}
Since $2s - n < 4K + 4$ by \eqref{odd}, and since $t \mapsto  \frac{t + 2n + 4k}{t - 4k}$ is strictly decreasing, we can estimate 
\[ \frac{2s + n + 4k}{2s-n - 4k} > \frac{4K + 4 + 2n + 4k}{4K + 4 - 4k} = \frac{2K + 2 + n + 2k}{2K + 2 - 2k}. \]
Hence to show \eqref{ineq}, it suffices to prove 
\begin{equation}
	\label{ineq2}
	\frac{2K + 2 + n + 2k}{2K + 2 - 2k} \frac{(2K + 1)^2 (2k + n -1) (4k + n +1)}{(2K + 2k + n + 1)^2(2k + 1) (4k + n -1)}  \geq  1
\end{equation}
for all integers $n \geq 2$, $K \geq 1$ and $0 \leq k \leq K$. 

The rest of the proof will be devoted to establishing \eqref{ineq2} by considering several cases separately.

Let us first assume that $k \leq K-1$. We write the left side of \eqref{ineq2} as 
\[ \frac{(4k + n +1)(2K + 2 + n + 2k)}{(2K + 2k + n + 1)^2} \cdot \frac{2k + n -1}{4k + n -1} \cdot \frac{(2K + 1)^2}{(2K + 2 - 2k)(2k + 1)}, \]
and notice that the first factor is a decreasing function of $n\geq 2$ (see Lemma \ref{lemma monotonicity in n} below; here is where the assumption $k \leq K-1$ enters) and the second factor is a decreasing function of $n\geq 2$. Thus, if $k \leq K-1$, it suffices to prove \eqref{ineq2} for $n=2$, which we write as  
\begin{equation}
	\label{ineq n=2}
	\frac{F(K,k)}{G(K,k)} :=  \frac{4k + 3}{4k+1}\cdot \frac{2K + 1}{2K + 2k + 3} \cdot   \frac{2K + 1}{2K - 2k + 2} \cdot  \frac{2K + 2k +4}{2K  + 2k+ 3}  \geq  1. 
\end{equation}
If $k = 0$, then for all $K\geq 1$
\[ \frac{F(K,0)}{G(K,0)} = 3 \cdot \frac{(2K+1)^2}{(2K+3)^2} \cdot \frac{K+2}{K+1}
\geq 3 \cdot \frac{3^2}{5^2} \cdot 1 > 1 \,, \]
so we may assume $k \geq 1$ in the following. To solve this case, we resort to explicit calculation. We compute
\[ \frac 12 F(K,k) = (16k+12)K^3+(16k^2+60k+36)K^2+(16k^2+48k+27)K+(4k+3)(k+2) \]
and 
\begin{align*} \frac{1}{2} G(K,k) &= (16k+4)K^3+(16k^2+68k+16)K^2+(-16k^3+28k^2+92k+21)K \\
	&\qquad -(k-1)(2k+3)^2(4k+1).
\end{align*}
Hence for every $k \geq 1$, dropping the positive constant term and using $k \leq K-1$, we get
\begin{align*}
	\frac{1}{4K} (F(K,k)-G(K,k)) &\geq 4K^2+(10-4k)K+(8k^3-6k^2-22k+3) \\
	&\geq  8k^3 - 6 k^2 - 12 k + 13 =: P(k) \,.
\end{align*}
We have $P(k) \geq 0$ for all $1 \leq k \leq K-1$ because $P'(k) = 24k^2 - 12k - 12 \geq 0$ for $k \geq 1$ and $P(1) = 3 > 0$. This finishes the proof of \eqref{ineq n=2}, and hence of \eqref{ineq2}, in the case $k \leq K-1$.

Let us finally give the proof of \eqref{ineq2} when $k = K$. If $K \geq 2$, we estimate the left side of \eqref{ineq2} by
\begin{equation}
	\label{ineq k=K}
	\frac{2K + 1}{2} \cdot \frac{2K + n -1}{4K + n -1} \cdot \frac{4K + n + 2}{4K + n + 1}  \geq \frac{2K + 1}{2} \cdot \frac12 \geq \frac{5}{4} > 1. 
\end{equation}
If $K = 1$ and $n \geq 3$, since $ \frac{2K + n -1}{4K + n -1}$ is increasing in $n$, the left side of \eqref{ineq k=K} can be estimated by 
\[ \frac{3}{2} \cdot \frac{2 + n-1}{4 + n-1} \cdot \frac{4 + n + 2}{4 + n + 1} > \frac{3}{2}  \cdot \frac{4}{6} = 1. \]
Finally, if $K = 1$ and $n = 2$, by a direct calculation the left side of \eqref{ineq k=K} equals $\frac{36}{35} > 1$. 

The proof of Theorem \ref{main2} is now complete.
\qed

\medskip

We finally prove the two lemmas that we referred to in the proof.

\begin{lemma}
	\label{lemma gegenbauer}
	The coefficients $c_k$ and $d_k$ in \eqref{t^2K, t^2K+1 expansion} are given by 
	\begin{align*}
		c_{k} &=\frac{2^{-2K}(2k+\frac{n-1}{2})\Gamma(2K+1) \Gamma(\frac{n-1}{2})}{\Gamma(K+k+\frac{n+1}{2})\Gamma(K-k+1)}, \\
		d_k &= \frac{2^{-2K-1} (2k + \frac{n+1}{2})\Gamma(2K+2) \Gamma(\frac{n-1}{2})}{\Gamma(K+k+\frac{n+3}{2})\Gamma(K-k+1)}.
	\end{align*} 
	Moreover, 
	\begin{equation}
		\label{L^2 norm}
		\int_{-1}^1 |C_{\ell}^{(\frac{n-1}{2})}(t)|^2 (1 - t^2)^\frac{n-2}{2} \, dt = \frac{\pi \, 2^{2-n} \, \Gamma(n-1+\ell)}{\ell! \, (\ell+\frac{n-1}{2}) \, \Gamma(\frac{n-1}{2})^2}.
	\end{equation} 
	In particular, identities \eqref{c_k d_k relation} and \eqref{L^2 norms relation} hold. 
\end{lemma}

\begin{proof}
	Formula \eqref{L^2 norm} is the special case $\alpha = \frac{n-1}{2}$ of the following general formula \cite[Table 18.3.1]{DLMF}
	\[ \int_{-1}^1 |C_{\ell}^{(\alpha)}(t)|^2 (1 - t^2)^{\alpha - \frac12} \, dt = \frac{\pi 2^{1-2\alpha}\Gamma(2\alpha+ \ell)}{\ell!(\ell+\alpha)\Gamma(\alpha)^2}. \]
	Observing that by change of variables $t = \omega_{n+1}$, 
	\[ \|C_{\ell}^{(\frac{n-1}{2})}(\omega_n) \|_{L^2(\mathbb S^n)}^2 = |\mathbb S^{n-1}|  \int_{-1}^1 |C_{\ell}^{(\frac{n-1}{2})}(t)|^2 (1 - t^2)^{\frac{n-2}{2}} \, dt, \]
	identity \eqref{L^2 norms relation} follows readily from \eqref{L^2 norm} by a direct computation. 
	
	To obtain the expression for $c_k$, recall that, at fixed $\alpha$, the Gegenbauer polynomials $C_\ell^{(\alpha)}$ are pairwise orthogonal on the space $L^2\left((-1, 1), (1-t^2)^{\alpha - \frac12} \, dt\right)$, see e.g. \cite[Table 18.3.1]{DLMF}. Integrating $t^{2K} = \sum_{l= 0}^K c_l C_{2l}^{(\frac{n-1}{2})} (t)$ against $C_{2k}^{(\frac{n-1}{2})} (t) (1 - t^2)^\frac{n-2}{2}$, we thus find 
	\[ \int_{-1}^1 t^{2K} C_{2k}^{(\frac{n-1}{2})} (t) (1 - t^2)^\frac{n-2}{2} \, dt = c_k \int_{-1}^1 |C_{\ell}^{(\frac{n-1}{2})}(t)|^2 (1 - t^2)^\frac{n-2}{2} \, dt. \]
	The integral on the right side is given in \eqref{L^2 norm} and the integral on the left side is \cite[Eq. 18.17.37]{DLMF}
	\[  \int_{-1}^1 t^{2K} C_{2k}^{(\frac{n-1}{2})} (t) (1 - t^2)^\frac{n-2}{2} \, dt = \frac{\pi 2^{2-n-2K}\Gamma(2k+n-1)\Gamma(2K+1)}{(2k)!\Gamma(\frac{n-1}{2})\Gamma(K+k+\frac{n+1}{2})\Gamma(K-k+1)} \,. \]
	The expression for $d_k$ follows analogously, using again \eqref{L^2 norm} and \cite[Eq.~18.17.37]{DLMF}. Finally, a direct computation gives identity \eqref{c_k d_k relation}.  
\end{proof}

\begin{lemma}
	\label{lemma monotonicity in n}
	Suppose that $0 \leq k \leq K-1$. Then $n \mapsto \frac{(4k + n +1)(2K + 2 + n + 2k)}{(2K + 2k + n + 1)^2}$ is increasing in $n \geq 2$. 
\end{lemma}

\begin{proof}
	More generally, let $f(n) := \frac{(a+n)(b+n)}{(c+n)^2}$ with $a,b,c \geq -1$, $a \leq c \leq b$. We claim that if 
	\begin{equation}
		\label{suff cond mon}
		c-a \geq b-c,
	\end{equation}
	then $f(n)$ is increasing in $n \geq 2$, unless when $a=b=c$. Applying this claim with $a = 4k+1$, $b = 2K + 2k + 2$ and $c = 2K + 2k + 1$, condition \eqref{suff cond mon} becomes $k \leq K - \frac12$ and the lemma follows. 
	
	To prove the claim, write
	\[ f(n) = \left( 1 - \frac{c-a}{c+n} \right) \left( 1 + \frac{b-c}{c+n} \right) = 1 + \frac{a + b- 2c}{c+n} - \frac{(c-a)(b-c)}{(c+n)^2}. \]
	The last summand is nonpositive by assumption, hence nondecreasing in $n$. The second summand is nondecreasing in $n$ if $a + b - 2c \leq 0$, which is just \eqref{suff cond mon}. If $a=b=c$ does not hold, then one of the summands is increasing. 
\end{proof}



\begin{thebibliography}{49}
	
	\bibitem{AiChWe} J.~Ai, K.-S.~Chou, J.~Wei, \textit{Self-similar solutions for the anisotropic affine curve shortening problem}. Calc.~Var.~Partial Differential Equations \textbf{13} (2001), no. 3, 311--337.
	
	\bibitem{Au} T.~Aubin, \textit{Probl\`emes isoperim\'etriques et espaces de Sobolev}. J. Differ. Geometry \textbf{11} (1976), 573--598.
	
	\bibitem{Au2} T.~Aubin, \textit{\'Equations diff\'erentielles non lin\'eaires et probl\`eme de Yamabe concernant la courbure scalaire}. J. Math. Pures Appl. (9) \textbf{55} (1976), no. 3, 269--296.
	
	\bibitem{Be1992} W.~Beckner, \textit{Sobolev inequalities, the Poisson semigroup, and analysis on the sphere $\Sph$}. Proc. Nat. Acad. Sci. U.S.A. \textbf{89} (1992), no. 11, 4816--4819.
	
	\bibitem{Be1993} W. Beckner, \textit{Sharp Sobolev inequalities on the sphere and the Moser-Trudinger inequality}. Ann. of Math. (2) \textbf{138} (1993), no. 1, 213–242.
		
	\bibitem{BiEg} G.~Bianchi, H.~Egnell, \textit{A note on the Sobolev inequality}. J. Funct. Anal. \textbf{100} (1991), no. 1, 18--24.
	
	\bibitem{BoDoNaSi} M.~Bonforte, J.~Dolbeault, B.~Nazaret, N.~Simonov, \textit{Stability in Gagliardo--Nirenberg--Sobolev inequalities: flows, regularity and the entropy method}. Preprint (2021), arXiv:2007.03674. 
	
	\bibitem{BrLi} H.~Brezis, E.~H.~Lieb, \textit{Sobolev inequalities with remainder terms}. J. Funct. Anal. \textbf{62} (1985), no. 1, 73--86.
	
	\bibitem{CaFrLi} E.~A.~Carlen, R.~L.~Frank, E.~H.~Lieb, \textit{Stability estimates for the lowest eigenvalue of a Schr\"odinger operator}. Geom. Funct. Anal. \textbf{24} (2014), no. 1, 63--84.
	
	\bibitem{CaLo} E. Carlen, M. Loss, \textit{Competing symmetries, the logarithmic HLS inequality and Onofri's inequality on $\Sph^n$}. Geom.~Funct.~Anal.~\textbf{2} (1992), no. 1, 90--104.
	
	\bibitem{CaDeDoFrHo} J.~A.~Carrillo, M.~Delgadino, J.~Dolbeault, R.~L.~Frank, F.~Hoffmann, \textit{Reverse Hardy--Littlewood--Sobolev inequalities}. J. Math. Pures Appl. (9) \textbf{132} (2019), 133--165.
	
	\bibitem{CaDeFrLe} J.~A.~Carrillo, M.~Delgadino, R.~L.~Frank, M.~Lewin, \textit{Fast diffusion leads to partial mass concentration in Keller--Segel type stationary solutions}. Preprint (2020), arXiv:2012.08586.
	
	\bibitem{CaCh} J.~S.~Case, S.--Y.~A.~Chang, \textit{On fractional GJMS operators}. Comm. Pure Appl. Math. \textbf{69} (2016), no.~6, 1017--1061. 
	
	\bibitem{ChGo} S.--Y.~A.~Chang, M.~d.~M.~Gonz\'alez, \textit{Fractional Laplacian in conformal geometry}. Adv. Math. \textbf{226} (2011), no.~2, 1410--1432.
	
	\bibitem{ChYa} S.--Y.~A.~Chang, R.~A.~Yang, \textit{On a class of non-local operators in conformal geometry}. Chin. Ann. Math. Ser. B \textbf{38} (2017), no.~1, 215--234.
	
	\bibitem{ChLiLuTa} L.~Chen, Z.~Liu, G.~Lu, C.~Tao, \textit{Reverse Stein--Weiss inequalities and existence of their extremal functions}. Trans. Amer. Math. Soc. \textbf{370} (2018), no.~12, 8429--8450.
		
	\bibitem{ChFrWe} S.~Chen, R.~L.~Frank, T.~Weth, \textit{Remainder terms in the fractional Sobolev inequality}. Indiana Univ. Math. J. \textbf{62} (2013), no.~4, 1381--1397. 
	
	\bibitem{ChLiOu} W. Chen, C. Li, B. Ou, \textit{Classification of solutions for an integral equation}. Comm. Pure Appl. Math. \textbf{59} (2006), no.~3, 330--343.
	
	\bibitem{ChXu} Y.~S.~Choi, X.~Xu, \textit{Nonlinear biharmonic equations with negative exponents}. J. Differential Equations \textbf{246} (2009), no.~1, 216--234.
	
	\bibitem{DLMF}	\textit{NIST Digital Library of Mathematical Functions}. http://dlmf.nist.gov/, Release 1.1.1 of 2021-03-15. F. W. J. Olver, A. B. Olde Daalhuis, D. W. Lozier, B. I. Schneider, R. F. Boisvert, C. W. Clark, B. R. Miller, B. V. Saunders, H. S. Cohl, and M. A. McClain, eds.
	
 	\bibitem{DoGuZh} J.~Dou, Q.~Guo, M.~Zhu, \textit{Subcritical approach to sharp Hardy--Littlewood--Sobolev type inequalities on the upper half space}. Adv. Math. \textbf{312} (2017), 1--45.
 	
	\bibitem{DoZh} J.~Dou, M.~Zhu, \textit{Reversed Hardy--Littewood--Sobolev inequality}. Int. Math. Res. Not. IMRN 2015, no.~19, 9696--9726.
	
	\bibitem{ExHaLo} P.~Exner, E.~M.~Harrell, M.~Loss, \textit{Optimal eigenvalues for some Laplacians and Schr\"odinger operators depending on curvature}. Mathematical results in quantum mechanics (Prague, 1998), 47--58, Oper. Theory Adv. Appl., \textbf{108}, Birkh\"auser, Basel, 1999.
	
	\bibitem{FiZh} A.~Figalli, Y.~R.--Y.~Zhang, \textit{Sharp gradient stability for the Sobolev inequality}. Preprint (2020), arXiv:2003.04037.
	
	\bibitem{Fr} R.~L.~Frank, \textit{Degenerate stability of some Sobolev inequalities}. Ann.~Inst.~H.~Poincar\'e~Anal.~Non Lin\'eaire, to appear. Preprint (2021), arXiv:2107.11608.
	
	\bibitem{FrKoTa} R.~L.~Frank, T. K\"onig, H.~Tang, \textit{Classification of solutions of an equation related to a conformal log Sobolev inequality}. Adv. Math. \textbf{375} (2020), 107395, 27 pp.
	
	\bibitem{FrLi} R. L. Frank, E. H. Lieb, \textit{A new, rearrangement-free proof of the sharp Hardy--Littlewood--Sobolev inequality}. Spectral theory, function spaces and inequalities, 55--67, Oper. Theory Adv. Appl. \textbf{219}, Birkh\"auser/Springer Basel AG, Basel, 2012. 
	
	\bibitem{Fu} N.~Fusco, \textit{The quantitative isoperimetric inequality and related topics}. Bull. Math. Sci. \textbf{5} (2015), no.~3, 517--607.
	
	\bibitem{FuMaPr} N.~Fusco, F.~Maggi, A.~Pratelli, \textit{The sharp quantitative isoperimetric inequality}. Ann. of Math. (2) \textbf{168} (2008), no.~3, 941--980.
	
	\bibitem{GrJeMaSp} C.~R.~Graham, R.~Jenne, L.~J.~Mason, G.~A.~Sparling, \textit{Conformally invariant powers of the Laplacian. I. Existence}. J. London Math. Soc. (2) \textbf{46} (1992), no.~3, 557--565. 
	
	\bibitem{GrZw} C.~R.~Graham, M.~Zworski, \textit{Scattering matrix in conformal geometry}. Invent.~Math.~\textbf{152} (2003), no.~1, 89--118.
	
	\bibitem{Ha} F.~Hang, \textit{On the higher order conformal covariant operators on the sphere}. Commun. Contemp. Math. \textbf{9} (2007), no.~3, 279--299.
	
	\bibitem{HaYa} F.~Hang, P.~C.~Yang, \textit{The Sobolev inequality for Paneitz operator on three manifolds}. Calc. Var. Partial Differential Equations \textbf{21} (2004), no.~1, 57--83.
	
	\bibitem{HaYa2} F.~Hang, P.~C.~Yang, \textit{Lectures on the fourth-order Q curvature equation}. Geometric analysis around scalar curvatures, 1--33, Lect. Notes Ser. Inst. Math. Sci. Natl. Univ. Singap., 31, World Sci. Publ., Hackensack, NJ, 2016.
	 
	\bibitem{HyWe} A.~Hyder, J.~Wei, \textit{Non-radial solutions to a bi-harmonic equation with negative exponent}. Calc. Var. Partial Differential Equations \textbf{58} (2019), no. 6, Paper No.~198, 11 pp.
	 
	\bibitem{LeMi} N.~A.~Lebedev, I.~M.~Milin, \textit{An inequality}. Vestnik Leningrad. Univ. \textbf{20} (1965), no.~19, 157--158. 
	
	\bibitem{Li04} Y.~Y.~Li, \textit{Remark on some conformally invariant integral equations: the method of moving spheres}. J. Eur. Math. Soc. (JEMS) \textbf{6} (2004), no.~2, 153--180.
		
	\bibitem{Li} E.~H.~Lieb, \textit{Sharp constants in the Hardy--Littlewood--Sobolev and related inequalities}. Ann. of Math. (2) \textbf{118} (1983), no.~2, 349--374.
	
	\bibitem{LiLo} E.~H.~Lieb, M.~Loss, \textit{Analysis. Second edition}. Graduate Studies in Mathematics \textbf{14}, American Mathematical Society, Providence, RI, 2001.
	
	\bibitem{NgNg} Q.~A.~Ng\^{o}, V.~H.~Nguyen, \textit{Sharp reversed Hardy--Littlewood--Sobolev inequality on $\R^n$}. Israel J. Math. \textbf{220} (2017), no.~1, 189--223.
	
	\bibitem{NiZh} Y. Ni, M. Zhu, \textit{Steady states for one-dimensional curvature flows}. Commun. Contemp. Math. \textbf{10} (2008), no.~2, 155--179.
	
	\bibitem{On} E. Onofri, \textit{On the positivity of the effective action in a theory of random surfaces}. Comm. Math. Phys. \textbf{86} (1982), no.~3, 321--326.
	
	\bibitem{OsPhSa} B.~Osgood, R.~Phillips, P.~Sarnak, \textit{Extremals of determinants of Laplacians}. J. Funct. Anal. \textbf{80} (1988), no.~1, 148--211.
	
	\bibitem{Rod} E. Rodemich, \textit{The Sobolev inequality with best possible constant}. Analysis Seminar Caltech, Spring 1966.
	
	\bibitem{Ro} G. Rosen, \textit{Minimum value for $c$ in the Sobolev inequality $\|\phi^3\|\leq c\|\nabla\phi\|^3$}. SIAM J. Appl. Math. \textbf{21} (1971), 30--32.
	
	\bibitem{StWe} E. M. Stein, G. Weiss, \textit{Introduction to Fourier Analysis on Euclidean Spaces}. PMS \textbf{32}, Princeton University Press, 1971.
	
	\bibitem{Ta} G. Talenti, \textit{Best constants in Sobolev inequality}. Ann. Mat. Pura Appl. \textbf{110} (1976), 353--372.
	
	\bibitem{YaZh} P.~Yang, M.~Zhu, \textit{On the Paneitz energy on standard three sphere}. ESAIM Control Optim. Calc. Var. \textbf{10} (2004), no.~2, 211--223.
	
	\bibitem{Zh} M.~Zhu, \textit{Prescribing integral curvature equation}. Differential Integral Equations \textbf{29} (2016), no.~9-10, 889--904. 
	
\end{thebibliography}
\end{document}